\documentclass[11.5pt]{amsart}
\usepackage{amsmath,amssymb,amsbsy,amsfonts,latexsym,amsopn,amstext,
                                               amsxtra,euscript,amscd,bm}
\usepackage[margin=1in]{geometry}                    
\usepackage{url}
\usepackage[colorlinks,linkcolor=blue,anchorcolor=blue,citecolor=blue]{hyperref}
\usepackage{color}
\usepackage{enumerate}
\usepackage{hhline}
\usepackage{graphicx, float}
\usepackage{subcaption}
\begin{document}

\newtheorem{theorem}{Theorem}[section]
\newtheorem{lemma}[theorem]{Lemma}
\newtheorem{conjecture}[theorem]{Conjecture}
\newtheorem{proposition}[theorem]{Proposition}
\newtheorem{corollary}[theorem]{Corollary}
\newtheorem{claim}[theorem]{Claim}
\newtheorem{example}[theorem]{Example}
\theoremstyle{definition}
\newtheorem{remark}[theorem]{Remark}
\newtheorem{definition}[theorem]{Definition}

\def\E{{\mathbb E}}
\def\F{{\mathbb F}}
\def\P{{\mathbb P}}
\def\R{{\mathbb R}}
\def\Z{{\mathbb Z}}
\def\C{{\mathbb C}}
\def\N{{\mathbb N}}
\def\O{{\mathbb O}}
\def\B{{\mathcal B}}
\def\G{{\mathcal G}}
\def\Hy{{\mathcal H}}
\def\M{{\mathcal M}}
\def\S{{\mathcal S}}
\def\T{{\mathcal T}}
\def\U{{\mathcal U}}

\newcommand{\rst}[1]{\ensuremath{{\mathbin\upharpoonright}\raise-.5ex\hbox{$#1$}}}


\title[First server's effect on the expected number of games in Tennis]{First server's effect on the expected number of games in Tennis}

 \author[A. Mohammadi] {Ali Mohammadi}

\email{ali.mohammadi.np@gmail.com}

\pagenumbering{arabic}

\begin{abstract}
We show that information on the first server influences the expected total number of games and margin in a tennis match under the standard assumption that each player’s serve point win probability remains constant, and identify the exact regions, in terms of these probabilities, in which this effect is non-negligible. We confirm numerically that this effect is bounded by at most $0.4$ games at both the set and match level. This translates, for example, to roughly a $9$ percent shift in the probability that a match exceeds $19.5$ games when the players' serve point win probabilities differ by $10$ percent.

We complement the analysis with an empirical comparison on professional match data, illustrating the adequacy of the constant-probability assumption for modelling the
total number of games.
\end{abstract}

\maketitle

\section{Background}
Write $\mathcal{P_\sigma}, \G_{\sigma}, \S_{\sigma}$ and $\M_{\sigma}$ for the probabilities of player $\sigma\in\{A, B\}$ winning a point, game, set and match against the other player in a tennis match, respectively. Under the assumption that $\mathcal{P}_{A}$ and $\mathcal{P}_{B}$ are constant throughout the match, different authors \cite{Bar,OMa,NewKel} have obtained expressions for the aforementioned probabilities, in terms of $\mathcal{P}_{A}$ and $\mathcal{P}_{B}$. In particular, it is shown in \cite{NewKel} that the probabilities of winning a set and match are independent of the order of serve. In this short note, we show this is not the case for the expected total number of games and margin in a best-of-three sets match and provide a numerical analysis, highlighting both the adequacy of the constant-probability assumption and the regions in which the order of serve has a non-negligible effect.

As shown in \cite[equation 5]{NewKel}, writing
\begin{equation}\label{eqn:gameprob}
G(x) = x^4 \cdot \left( 15 - 4x- \frac{10x^2}{1-2x(1-x)}\right),
\end{equation}
we have $\G_{A} = G(\mathcal{P}_{A})$.

Let $\S^{\sigma}_{nm}$, $\sigma \in \{A, B\}$, denote the probability that a set ends with score $n-m$, where $n$ is the number of games won by the first server $\sigma$. See Appendix for the corresponding equations. By \cite[equations 18-20 and 22]{NewKel}
\begin{equation}\label{NK18}
\S_{60}^{A} = \S_{06}^{B}, \quad \S_{75}^{A} = \S_{57}^{B}, \quad
\S_{76}^{A} = \S_{67}^{B}.
\end{equation}
By \cite[equations 23-24]{NewKel}, we also have
\begin{equation}\label{NK23}
\S_{61}^{A}+\S_{62}^{A}=\S_{16}^{B}+\S_{26}^{B},\quad \S_{63}^{A}+\S_{64}^{A}=\S_{36}^{B}+S_{46}^{B}.
\end{equation}

We formulate our main results on the expected number of games and margins in terms of $\G_{A}$ and $\G_{B}$. The following preliminary results show they can be reformulated in terms of $\mathcal{P}_A$ and $\mathcal{P}_B$. Furthermore, all results require $\mathcal{G}_A+\mathcal{G}_B>1$ which is in practice almost always satisfied. We assume, without loss of generality, that $\G_{A}\geq\G_{B}$.
\begin{lemma}\label{lem:G-monotone}
The function $G(x)$, defined by \eqref{eqn:gameprob}, is strictly increasing for $0<x<1$. In particular,
\[
\G_{A}>\G_{B}
\quad\Longleftrightarrow\quad
\mathcal{P}_{A}>\mathcal{P}_{B}.
\]
\end{lemma}

\begin{proof}
We have
\[
G^{'}(x)
=
\frac{20x^3(1-x)^3(3+4x(x-1))}
{(2(x-1)x+1)^2}.
\]
The claim follows noting that each factor on the right-hand side is positive for $0<x<1$.
\end{proof}

\begin{corollary} If \(\mathcal{P_A}+\mathcal{P_B}>1\), then \(\mathcal{G}_{A}+\mathcal{G}_B>1\).
\end{corollary}
\begin{proof}
Let $G(x)$
be the function \eqref{eqn:gameprob}.
By Lemma~\ref{lem:G-monotone}, $G$ is strictly increasing on \((0,1)\). Then, using $G(1-x)=1-G(x)$, the assumption
$\mathcal{P}_{A}>1-\mathcal{P}_{B}$ and since $G$ is strictly increasing,
\[
\G_{A}=G(\mathcal{P}_{A})>G(1-\mathcal{P}_{B})=1-G(\mathcal{P}_{B})=1-\G_{B},
\]
as required.
\end{proof}

\section{Order of serve effect at set level}
\begin{proposition}\label{prop:set-length-shift}
Assume $\G_A + \G_B > 1$. Write $T^\text{set}_A$ for the expected number of games in a set, where $A$ serves first. Then $T^\text{set}_A < T^\text{set}_B$ if $\G_A>\G_B$.
\end{proposition}

\begin{proof}
Let $\pi_n^{A}$ denote the probability that a set, with initial server $A$, ends with a total of $n$ games played. Similarly define $\pi_n^{B}$. Then, clearly $\pi_{n+m}^{A}=\S_{nm}^{A}+\S_{mn}^{A}$ and similarly for $B$. By \eqref{NK23},
\begin{align*}
\pi_7^{A}-\pi_7^{B}
&=(\S_{61}^{A}+\S_{16}^{A})-(\S_{61}^{B}+\S_{16}^{B}) \\
&= -(\S_{62}^{A}+\S_{26}^{A})+(\S_{62}^{B}+\S_{26}^{B}) \\
&=-(\pi_8^{A}-\pi_8^{B}).
\end{align*}
Similarly,
\[
\pi_9^{A}-\pi_9^{B}=-(\pi_{10}^{A}-\pi_{10}^{B}).
\]

A direct computation gives
\[
\pi_7^{A}-\pi_7^{B}
=
3(\G_A-\G_B)(\G_A+\G_B-1)\bigl(\G_A+\G_B-2\G_A\G_B\bigr)\bigl(1-\G_A-\G_B+2\G_A\G_B\bigr).
\] 
Every factor on the right-hand side is strictly positive under our hypotheses, so
\[
\pi_7^{A}-\pi_7^{B}>0, \quad \text{and thus} \quad \pi_8^{A}-\pi_8^{B}<0.
\]
We also have
\begin{align*}
\pi_9^{A}-\pi_9^{B}
&=
4(\G_A-\G_B)(\G_A+\G_B-1)\bigl(1-\G_A-\G_B+2\G_A\G_B\bigr) \\
&\qquad\times \Big(1 - 2 \G_A + \G_A^2 - 2 \G_B + 9 \G_A \G_B - 7 \G_A^2 \G_B + \G_B^2 - 7 \G_A \G_B^2 + 
 7 \G_A^2 \G_B^2\Big).
\end{align*}
The first three factors on the right-hand side are strictly positive under the stated hypotheses. To show the last factor, denoted by $L$, is positive, put $s=\G_A+\G_B$, and $p=\G_A\G_B$. Then
\[
L=(1-s)^2+7p(1-s)+7p^2.
\]
Since $\G_A+\G_B>1$, write $u=s-1>0$. Then $L=u^2-7pu+7p^2$.
 Moreover, because $0<\G_A,\G_B<1$, we have $(1-\G_A)(1-\G_B)\ge 0$,
and hence $s-1\le \G_A\G_B=p$. Thus $0<u\le p$. The function
\[
u\mapsto u^2-7pu+7p^2
\]
is decreasing on $[0,p]$, since $2u-7p<0$ there. Therefore $L\ge p^2>0$. Hence
\[
\pi_9^{A}-\pi_9^{B}>0 \quad \text{and thus} \quad \pi_{10}^{A}-\pi_{10}^{B}<0.
\]
By equation \eqref{NK18}, one has
\[
\pi_6^{A}=\pi_6^{B},\qquad
\pi_{12}^{A}=\pi_{12}^{B},\qquad
\pi_{13}^{A}=\pi_{13}^{B}. 
\]
Therefore
\begin{align*}
T^\text{set}_A - T^\text{set}_B
&=\sum_{n=6}^{13} n\bigl(\pi_n^{A}-\pi_n^{B}\bigr) \\
&=7(\pi_7^{A}-\pi_7^{B})+8(\pi_8^{A}-\pi_8^{B}) \\
&\qquad +9(\pi_9^{A}-\pi_9^{B})+10(\pi_{10}^{A}-\pi_{10}^{B}) \\
&=-(\pi_7^{A}-\pi_7^{B})-(\pi_9^{A}-\pi_9^{B}) \\
&<0,
\end{align*}
as required.
\end{proof}
\begin{proposition}\label{prop:set-margin-shift}
Let $H$ denote the the number of games won by $A$ minus games won by $B$ in a given set and write $H^\text{set}_A=\E(H\mid S=A)$, and $H^\text{set}_B=\E(H\mid S=B)$, where $S$ is the first server. Assume that $\G_A+\G_B>1$. Then $H^\text{set}_A>H^\text{set}_B$.
\end{proposition}

\begin{proof}
Let
\[
\Delta_{nm}=\S_{nm}^{A}-\S_{mn}^{B}.
\]
Then
\begin{align}
H_A^{\mathrm{set}}-H_B^{\mathrm{set}}
=
\sum_{k=0}^{4} (6-k)\bigl(\Delta_{6k}-\Delta_{k6}\bigr)
+2(\Delta_{75}-\Delta_{57})
+(\Delta_{76}-\Delta_{67}).
\label{eq:margin-expand}
\end{align}
Using \eqref{NK18}, we have
\[
6(\Delta_{60}-\Delta_{06})
+2(\Delta_{75}-\Delta_{57})
+(\Delta_{76}-\Delta_{67})
=0.
\]
A direct computation of the remaining terms gives
\begin{align*}
&5(\Delta_{61}-\Delta_{16})
 +4(\Delta_{62}-\Delta_{26}) \notag
 +3(\Delta_{63}-\Delta_{36})
 +2(\Delta_{64}-\Delta_{46})=\\ &
-\,(\G_A + \G_B -1)\,(1-\G_A-\G_B+2\G_A\G_B)
\Big(
-4 \G_A + 5 \G_A^2 - 4 \G_A^3 
- 4 \G_B + 24 \G_A \G_B - 50 \G_A^2\\  &\G_B + 36 \G_A^3 \G_B
+ 5 \G_B^2 - 50 \G_A \G_B^2 + 122 \G_A^2 \G_B^2 - 84 \G_A^3 \G_B^2
- 4 \G_B^3 + 36 \G_A \G_B^3 - 84 \G_A^2 \G_B^3 + 56 \G_A^3 \G_B^3
\Big).
\end{align*}
Write $L(\G_A,\G_B)$ for the last factor. Since $L$ is symmetric, put $s=\G_A+\G_B$, and $p=\G_A \G_B$.
A direct expansion gives
\[
L
=
-4s-38ps-84p^2s+5s^2+36ps^2-4s^3
-10p+12ps .
\]
Now set
\[
u=s-1=\G_A+\G_B-1,
\qquad
q=1-s+2p
=1-\G_A-\G_B+2\G_A \G_B .
\]
Then 
\[
L
=
-(1+u)\bigl(3+q+21q^2+4u+24qu+7u^2\bigr).
\]
Under the assumptions $\G_A+\G_B>1$ and $0<\G_A,\G_B<1$, we have
\[
u>0,
\qquad
q=\G_A \G_B+(1-\G_A)(1-\G_B)>0.
\]
Hence every term in the second factor is positive, and therefore $L(\G_A,\G_B)<0$, as required.
\end{proof}
\section{Order of serve effect at match level}
The following claim provides a reduction of the match-level dependence on the order of serve to the initial server of the first set. In particular, it encodes the evolution of serve order across sets in terms of first set quantities, allowing totals and margin comparisons at the match level to be expressed through first set asymmetries.
\begin{claim}\label{claim:set_transition}
Define the one-set transition probabilities
\[
\alpha_\ell = \P(W=A,\ L=\ \ell \mid S=A),\quad
\beta_\ell = \P(W=B,\ L=\ \ell \mid S=A),
\]
\[
\gamma_\ell = \P(W=A,\ L=\ \ell \mid S=B),\quad
\delta_\ell = \P(W=B,\ L=\ \ell \mid S=B),
\quad \ell\in\{\mathrm{e},\mathrm{o}\}.
\]
where $L$ is the set length parity, writing $e$ for even and $o$ for odd, $S$ is the initial server and $W$ the set winner. Write $p$ for the probability that $A$ wins a set. Then
\[
\alpha_e+\alpha_o=\gamma_e+\gamma_o=p,\qquad
\beta_e+\beta_o=\delta_e+\delta_o=1-p.
\]

Let \(S_j\in\{A,B\}\) denote the first server of set \(j\), let \(N\in\{2,3\}\) be the
number of sets. The following quantities determine the transition structure:
\[
q_A = \P(S_2=A \mid S_1=A)=\alpha_e+\beta_e,\qquad
q_B = \P(S_2=A \mid S_1=B)=\gamma_o+\delta_o,
\]
and 
\[
\rho_A = \P(S_3=A,\ N=3 \mid S_1=A)
=2\alpha_e\beta_e+\alpha_o\delta_o+\beta_o\gamma_o,
\]
\[
\rho_B = \P(S_3=A,\ N=3 \mid S_1=B)
=\gamma_e\delta_o+\gamma_o\beta_e+\delta_e\gamma_o+\delta_o\alpha_e.
\]

Setting $x=\alpha_e-\gamma_o$, and $y=\beta_e-\delta_o$, we have
\[
q_A-q_B = x+y,\qquad \rho_A-\rho_B = 2xy.
\]
\end{claim}
\begin{proof}
A third set is played exactly when the first two sets are split. Under
$S_1=A$, the only ways that set $3$ is played and started by $A$ are
\[
(A\text{ wins even},\,B\text{ wins even}),
\quad
(B\text{ wins even},\,A\text{ wins even}),
\]
\[
(A\text{ wins odd},\,B\text{ wins odd}),
\quad
(B\text{ wins odd},\,A\text{ wins odd}),
\]
giving the equation for $\rho_A$. We similarly derive $\rho_B$. Also,
\begin{align*}
\rho_A-\rho_B
&=2\alpha_e\beta_e+\alpha_o\delta_o+\beta_o\gamma_o
-\gamma_e\delta_o-\gamma_o\beta_e-\delta_e\gamma_o-\delta_o\alpha_e\\
&=2\alpha_e\beta_e
+(p-\alpha_e)\delta_o
+(1-p-\beta_e)\gamma_o \\
&\qquad
-(p-\gamma_o)\delta_o
-\gamma_o\beta_e
-(1-p-\delta_o)\gamma_o
-\delta_o\alpha_e\\
&=2\alpha_e\beta_e
-2\alpha_e\delta_o
-2\beta_e\gamma_o
+2\gamma_o\delta_o\\
&=2(\alpha_e-\gamma_o)(\beta_e-\delta_o)
=2xy.
\end{align*}
\end{proof}

\begin{proposition}\label{prop:match_exp_totals}
Let $T$ denote the total number of games in a best-of-three sets match between $A$ and $B$. Assume $\G_A + \G_B > 1$. Then, writing $T^\text{match}_A= \E(T\mid S_1=A)$ and $T^\text{match}_B = \E(T\mid S_1=B)$, we have $T^\text{match}_A<T^\text{match}_B$ if $\G_A>\G_B$.

\end{proposition}
\begin{proof}
Write $N\in \{2, 3\}$ for the number of sets played and $N_A$ for those started by $A$. Let $T_1$ denote the total number of games won by $A$ in the first set, $T^\text{set}_A = \E(T_1\mid S_1=A)$,
and similarly define $T^\text{set}_B$. Then
\[
\E(T\mid S_1=\sigma)
=
T^\text{set}_B\cdot\,\E(N\mid S_1=\sigma)
+
(T^\text{set}_A-T^\text{set}_B)\,\E(N_A\mid S_1=\sigma),
\]
for $\sigma\in\{A, B\}$. By \cite[Section 3.3]{NewKel}, probability of winning a set is independent of the first server, and so $\E(N\mid S_1=A)=\E(N\mid S_1=B)$, giving
\[
T^\text{match}_A-T^\text{match}_B
=
(T^\text{set}_A-T^\text{set}_B)(M_A-M_B),
\]
where
\[
M_A=\E(N_A\mid S_1=A),\qquad M_B=\E(N_A\mid S_1=B).
\]
We compute \(M_A-M_B\). Using Claim~\ref{claim:set_transition}, and $M_A = 1 + q_A + \rho_A$,
$M_B = q_B + \rho_B$.

Hence
\[
M_A-M_B = 1 + (q_A-q_B) + (\rho_A-\rho_B).
\]
Writing $x=\alpha_e-\gamma_o$, and $y=\beta_e-\delta_o$, Claim~\ref{claim:set_transition} gives
\[
M_A-M_B = 1 + x + y + 2xy.
\]
Since $F(x,y)=1 + x + y + 2xy$ is affine in each variable separately, its minimum over the rectangle $[-p,p]\times[-(1-p),1-p]$ is attained at a corner. At the four corners, $F$ takes the values $2p(1-p)$, $2(1-p)^2$, $2p^2$, and $2+2p(1-p)$,
all of which are strictly positive for $0<p<1$. Hence $M_A-M_B>0$.
Since, by Proposition~\ref{prop:set-length-shift}, \(T^\text{set}_A-T^\text{set}_B<0\), it follows that $T^\text{match}_A<T^\text{match}_B$
as required.
\end{proof}

\begin{proposition}\label{prop:match_exp_handicaps}
Assume that $\G_A+\G_B>1$. Let $H^\text{match}_A=\E(H\mid S_1=A)$, and  $H^\text{match}_B=\E(H\mid S_1=B)$, where $H$ is the difference between number of games won by $A$ and $B$. 
Then $H^\text{match}_A>H^\text{match}_B$.
\end{proposition}

\begin{proof}
Let $H_1$ denote the difference between number of games won by $A$ and $B$ in the first set and write
\[
H^\text{set}_A=\E(H_1\mid S=A),\qquad H^\text{set}_B=\E(H_1\mid S=B).
\]
As in Proposition~\ref{prop:match_exp_totals}, writing $N$ for the number of sets played and $N_A$ for those started by $A$,
\[
\E(H\mid S_1=\sigma)
=
H^\text{set}_B\,\E(N\mid S_1=\sigma)
+
(H^\text{set}_A-H^\text{set}_B)\,\E(N_A\mid S_1=\sigma).
\]
Since \(\E(N\mid S_1=A)=\E(N\mid S_1=B)\), we obtain
\[
H^\text{match}_A-H^\text{match}_B = (H^\text{set}_A-H^\text{set}_B)(M_A-M_B),
\]
where $M_A=\E(N_A\mid S_1=A)$, and $M_B=\E(N_A\mid S_1=B)$.
By Proposition~\ref{prop:match_exp_totals}, \(M_A-M_B>0\). Hence, since by Proposition~\ref{prop:set-margin-shift}, \(H^\text{set}_A-H^\text{set}_B>0\), it follows $H^\text{match}_A-H^\text{match}_B>0$,
as required.
\end{proof}

\section{Numerical analysis}
\subsection{Quantitative analysis of $T^\text{match}_A-T^\text{match}_B$ and $H^\text{match}_A-H^\text{match}_B$}

Figures~\ref{fig:total_heatmaps} and~\ref{fig:handicap_heatmaps} quantify the dependence of totals and margin on the initial server, providing heatmaps of $T^\text{match}_A-T^\text{match}_B$ and $H^\text{match}_A-H^\text{match}_B$ from Propositions~\ref{prop:match_exp_totals} and \ref{prop:match_exp_handicaps} respectively. 
 The identity of the initial server
induces a shift of up to about $0.4$ games in expectation. In particular, if the stronger player serves first, the expected total is reduced by this amount. For margin, this separation also reaches approximately
$0.4$ games.

Furthermore, since the transition from set-level to match-level quantities introduces only
a negligible adjustment, the serve-order effect is comparatively more significant in first-set
markets than in full-match markets.

\begin{figure}[H]
\centering

\begin{subfigure}{0.48\textwidth}
    \includegraphics[width=\linewidth]{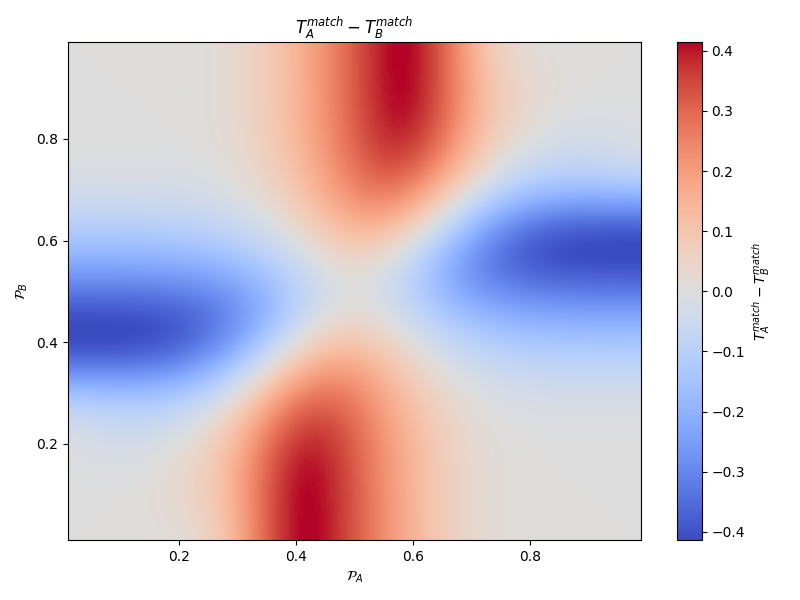}
    \caption{Match-level effect on totals}
\end{subfigure}
\hfill
\begin{subfigure}{0.48\textwidth}
    \includegraphics[width=\linewidth]{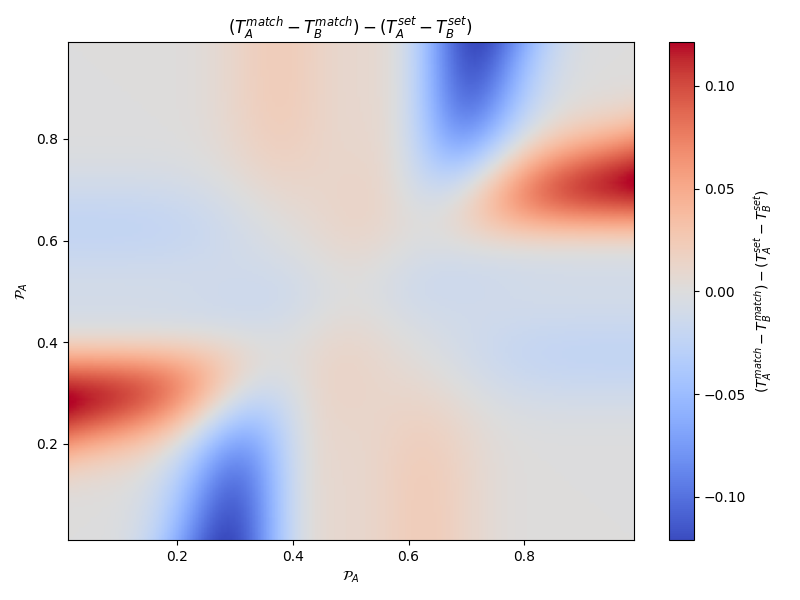}
    \caption{Match and set totals effect difference}
\end{subfigure}

\caption{Expected total games differential caused by serve order.}
\label{fig:total_heatmaps}
\end{figure}
\begin{figure}[H]
\centering

\begin{subfigure}{0.48\textwidth}
    \includegraphics[width=\linewidth]{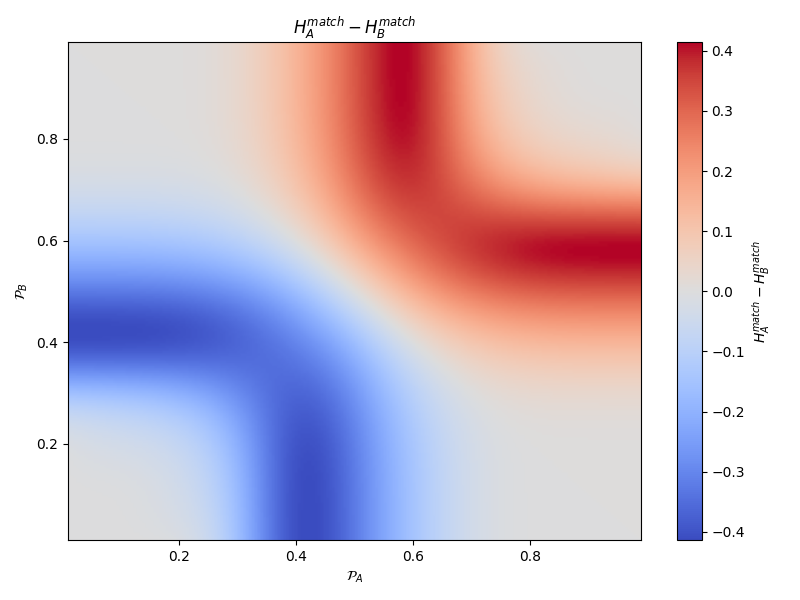}
    \caption{Match-level effect on margin}
\end{subfigure}
\hfill
\begin{subfigure}{0.48\textwidth}
    \includegraphics[width=\linewidth]{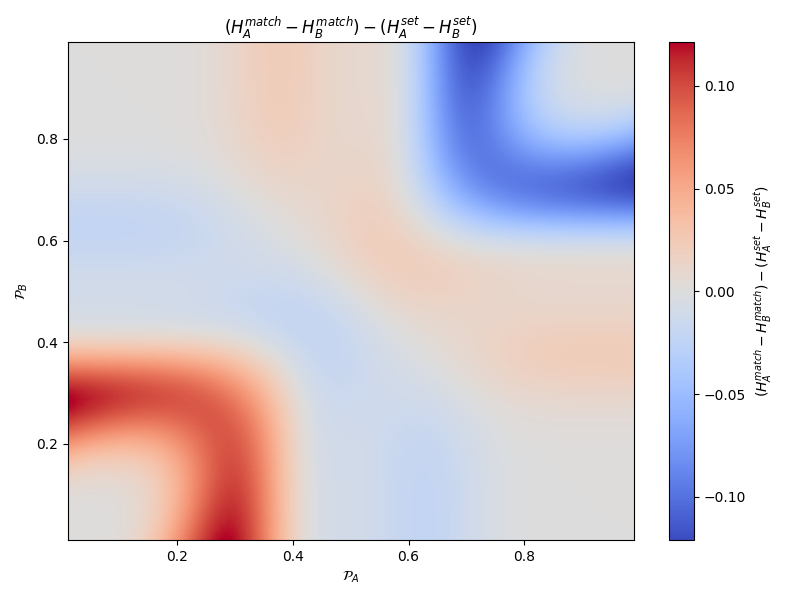}
    \caption{Match and set margin effect difference}
\end{subfigure}

\caption{Expected margin differential  caused by serve order.}
\label{fig:handicap_heatmaps}
\end{figure}

\subsection{Serve order effect on probabilities of exceeding total games lines}

To quantify the effect of service order, we also evaluated $\mathbb{P}(T>n\mid S_1=A)$ and $\mathbb{P}(T>n\mid S_1=B)$
over the grid $0.50 \leq \mathcal{P}_B \leq \mathcal{P}_A \leq 0.70$, for total games lines \(18.5,\ldots,22.5\). For each line, we computed the mean and maximum values of $\mathbb{P}(T>n\mid S_1=A)-\mathbb{P}(T>n\mid S_1=B)$
across the parameter grid, as shown in Figure~\ref{fig:over_prob_difference}. We note that the maximums at $19.5$ and $21.5$ game lines occur at $(\mathcal{P}_A, \mathcal{P}_B) = (0.7, 0.6)$ and $(0.7, 0.55)$, respectively.

\begin{figure}[ht]
\centering
\includegraphics[width=0.6\textwidth]{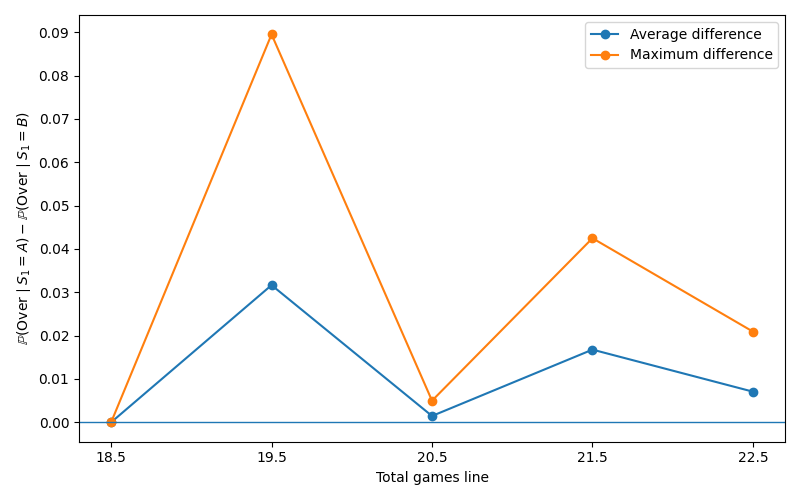}
\caption{Mean and maximum of $\mathbb{P}(T>n\mid S_1=A)-\mathbb{P}(T>n\mid S_1=B)$ for $n\in[18.5, 22.5]$ over the grid $0.50 \leq \mathcal{P}_B \leq \mathcal{P}_A \leq 0.70$.}
\label{fig:over_prob_difference}
\end{figure}

\subsection{Total games model residual analysis}
Figure~\ref{fig:model_heatmaps} presents the residuals of the constant-probability model for total games across approximately $30{,}000$ matches for each of ATP and WTA tours~\cite{SackmannTennisData}. In regions where $\mathcal{P}_A$ and $\mathcal{P}_B$ are sufficiently separated, the residuals are tightly concentrated around zero. In particular, the serve-order effects identified in this paper occur predominantly in regions where the residuals are negligible, indicating that these asymmetries are well captured within the constant-probability framework.

In contrast, the largest residuals occur in a neighbourhood of the diagonal $\mathcal{P}_A \approx \mathcal{P}_B$.
This regime corresponds to balanced matches, where the assumption that
$\mathcal{P}_A$ and $\mathcal{P}_B$ remain constant becomes most fragile: small temporal variations likely accumulate over many games, leading to systematic deviations
from the model. This behaviour is
consistent with the sigmoid shape of the underlying probability functions, where the gradient is
maximal near $\mathcal{P}_A = \mathcal{P}_B$, and hence small perturbations in the inputs induce
disproportionately large changes in outcomes. See, for example,
\cite[Figures~1 and~2]{NewKel} for the sigmoid curves corresponding to games and set winning probabilities.

This phenomenon is reasonably consistent across both ATP and WTA data, indicating that the breakdown is
structural rather than dataset-specific.

\begin{figure}[H]
\centering
\begin{subfigure}{0.48\textwidth}
    \includegraphics[width=\linewidth]{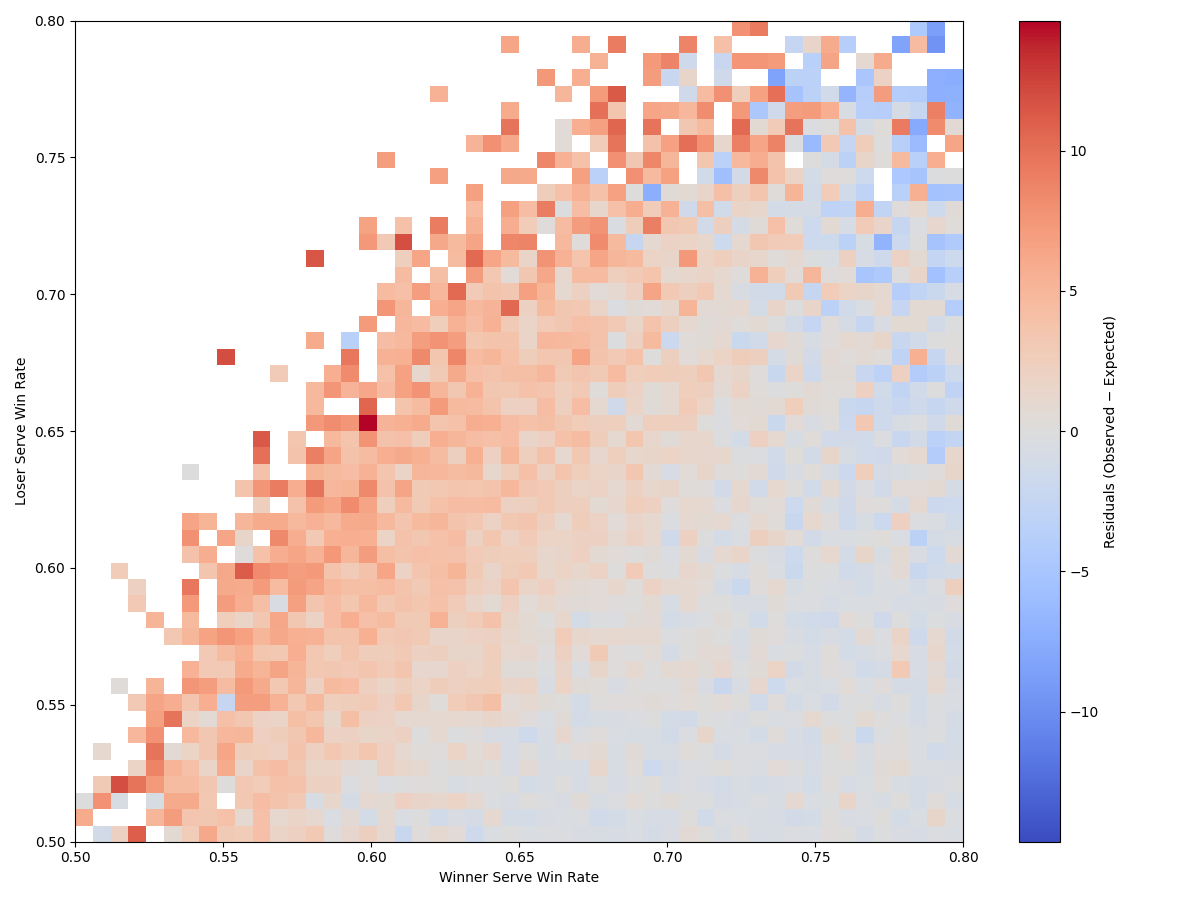}
    \caption{Model residuals for ATP data}
\end{subfigure}
\hfill
\begin{subfigure}{0.48\textwidth}
    \includegraphics[width=\linewidth]{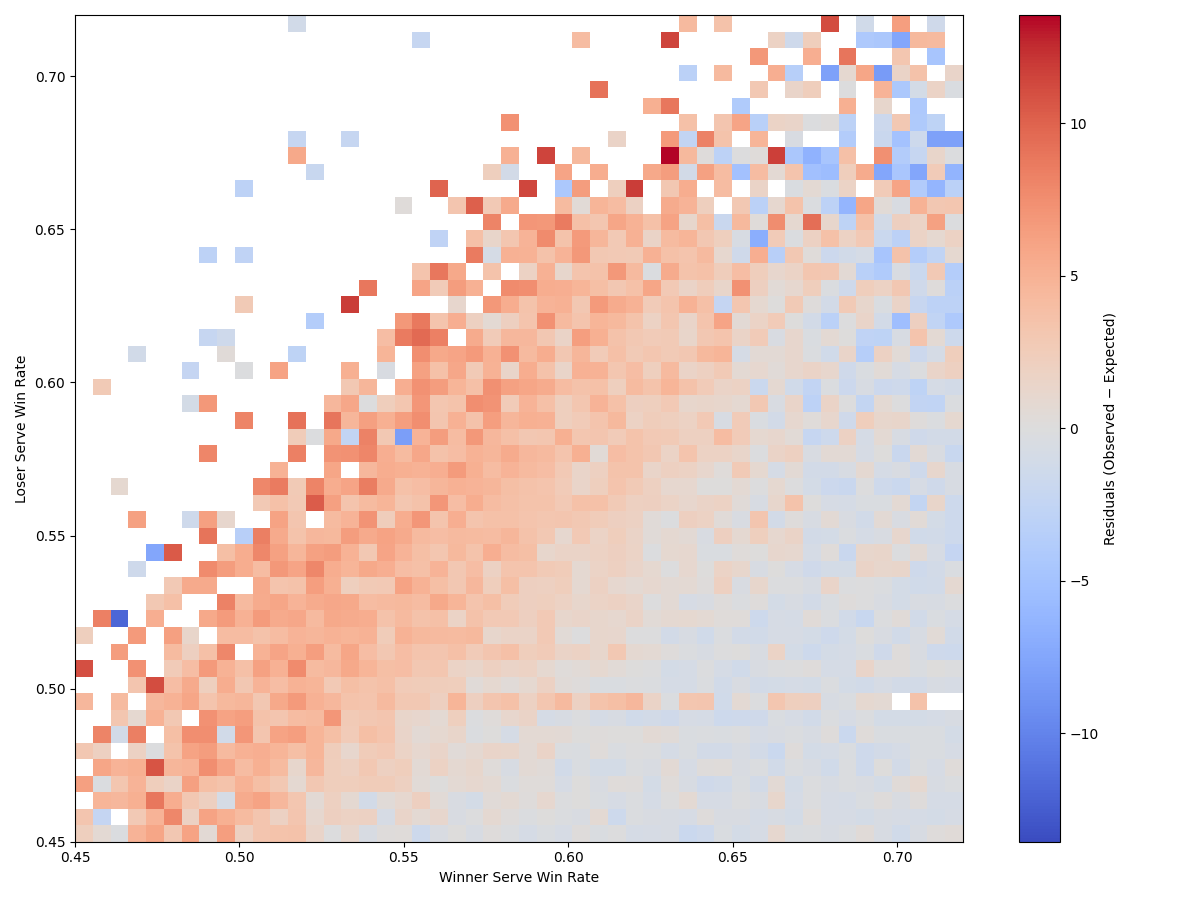}
    \caption{Model residuals for WTA data}
\end{subfigure}

\caption{Total games model residual structures.}
\label{fig:model_heatmaps}
\end{figure}

\subsection{Determining the first server in our dataset}

The dataset used in this paper~\cite{SackmannTennisData} does not record the identity of the first server. We infer the order of serve from the reported scores and the number of breaks of serve. (The dataset records the number of service games for each player, although since there is a large number of missing values for the WTA dataset, we resort to using break statistics.) 

Let \(G_W\) and \(G_L\) denote the total numbers of games won by the match winner and loser respectively, deduced from the score line, and let \(SG_W\) and \(SG_L\) denote their corresponding numbers of service games. Then $SG_W = G_W+ \text{breaks}_W-\text{breaks}_L$. Noting that a tiebreak is counted as an additional game for the purposes of service order alternation~\cite{usta_rules}.

The inference procedure is based on the observation that a set of length $n$ contains $\left\lceil n/2 \right\rceil$
service games for the player who serves first and $\left\lfloor n/2 \right\rfloor$
service games for the other player. 

Serve order is then propagated across sets using the parity of the number of games played in each set. We assume initially that the winner served first in the opening set. Using the parity rule above, we propagate the serve order through the match and compute the implied service-game totals $\widehat{\mathrm{SG}}_W,\widehat{\mathrm{SG}}_L$. If unequal, 
these are compared with the observed totals $\mathrm{SG}_W,\mathrm{SG}_L$ to determine the identity of the first server. Table~\ref{tab:first_server_inference} summarises the results based on matches spanning 2010-2024.

\begin{table}[h]
\centering
\begin{tabular}{lrrrr}
\hline
Tour & Total Matches & Determined & Indeterminate & Determined (\%) \\
\hline
ATP & 32,059 & 14,324 & 17,735 & 44.68 \\
WTA & 38,409 & 16,161 & 22,248 & 42.08 \\
\hline
\end{tabular}
\caption{First server inference results.}
\label{tab:first_server_inference}
\end{table}

\subsection{Residual analysis conditional on initial server}

For each match, the residual $R$ is calculated as the observed total number of games minus the expected number of games from the constant serve win probability model, which is computed using the inferred initial server. Table~\ref{tab:serve_order_residuals} summarises the residual distributions according to the identity of the first server.

\begin{table}[ht]
\centering
\caption{Statistics for total games residuals according to initial server.}
\label{tab:serve_order_residuals}

\begin{tabular}{llrrrrrrrr}
\hline
Tour & First Server & $n$ & Mean & Std & SE & Median & Q25 & Q75 \\
\hline
ATP & L & 5135 & 1.847 & 4.221 & 0.059 & 0.754 & -1.429 & 5.582 \\
ATP & W & 9189 & 0.201 & 4.080 & 0.043 & -0.713 & -2.603 & 2.907 \\
WTA & L & 6677 & 1.332 & 3.744 & 0.046 & 0.070 & -1.353 & 4.179 \\
WTA & W & 9484 & 0.582 & 3.677 & 0.038 & -0.441 & -1.854 & 2.831 \\
\hline
\end{tabular}
\end{table}
The results reveal a clear dependence on serve order. Matches in which the loser serves first tend to last relatively longer than predicted by the model, compared to the matches in which the winner serves first. 

 We also compute the average residual, ignoring the initial server's identity. The results suggest a modest systematic tendency for the model to underestimate match duration. As Figure~\ref{fig:model_heatmaps} suggests, this is largely due to inclusion of matches, where $\mathcal{P}_W \approx \mathcal{P}_L$.
\begin{table}[ht]
\centering
\caption{Statistics for total games residuals.}
\label{tab:mean_serve_order_residuals}
\begin{tabular}{lrr}
\hline
Tour & Mean Residual & SE \\
\hline
ATP & 0.792 & 0.035 \\
WTA & 0.892 & 0.029 \\
\hline
\end{tabular}
\end{table}

\subsection{Association between serve superiority and the likelihood to open a match}

To investigate whether relative serve superiority is associated with the identity of the first server, a logistic regression model was fitted with response variable
\[
Y=
\begin{cases}
1,& \text{if the winner served first},\\
0,& \text{if the loser served first},
\end{cases}
\]
and covariate $\Delta_S = \mathcal{P}_W-\mathcal{P}_L$,
where $\mathcal{P}_W$ and $\mathcal{P}_L$ denote the serve win probabilities of the winner and loser respectively. Thus, positive coefficients indicate that larger values of $\Delta_S$ are associated with an increased probability that the eventual winner served first. Table~\ref{tab:logistic_serve_order} summarises the fitted models.

\begin{table}[ht]
\centering
\caption{Logistic regression results relating $\Delta_S$ to the probability that the winner served first.}
\label{tab:logistic_serve_order}
\begin{tabular}{lrrrrrr}
\hline
Tour & $\beta_1$ & SE & Odds Ratio & 95\% CI Lower & 95\% CI Upper & $p$-value \\
\hline
ATP & 2.027 & 0.196 & 7.591 & 5.168 & 11.152 & $5.16\times10^{-25}$ \\
WTA & 0.587 & 0.170 & 1.798 & 1.289 & 2.507 & $5.42\times10^{-4}$ \\
\hline
\end{tabular}
\end{table}

Here 
\[
\beta_1
=
\frac{\partial}{\partial \Delta_S}
\log\!\left(
\frac{\Pr(Y=1)}
     {\Pr(Y=0)}
\right).
\]For both ATP and WTA matches, the coefficient of $\Delta_S$ is positive and statistically significant. This indicates that higher values of $\Delta_S$ are associated with an increased probability that the eventual winner served first. This is consistent with the hypothesis that stronger players are more inclined to serve first in order to exert early score-line pressure, whereas weaker players may prefer to receive first to avoid the psychological disadvantage associated with falling behind after an early break.

The effect is substantially stronger in the ATP data. The estimated coefficient of $2.027$ corresponds to an odds ratio of $7.59$, implying a strong positive association between serve superiority and the probability that the winner served first. In contrast, the corresponding WTA odds ratio is $1.80$, indicating a weaker but still statistically significant relationship.

\begin{remark}
Write $M(x,y)$ for the function \cite[equation 48]{NewKel}, such that the probability of $A$ winning a match against $B$ is $\M_{AB} = M(\mathcal{P}_{A}, \mathcal{P}_{B})$. A useful observation in \cite{KnoSpaMad} is that $M(\mathcal{P}_{A}, \mathcal{P}_{B}) \approx M(\mathcal{P}_{A} + \varepsilon, \mathcal{P}_{B}+ \varepsilon)$ for sufficiently small values of $\varepsilon >0$ (see \cite[Figure~1]{KnoSpaMad}). This justifies the authors' use of the approximation
\[
\M_{AB} \approx \frac{M(0.6 + \Delta_{AB}, 0.6) + M(0.6, 0.6-\Delta_{AB})}{2},
\]
as long as $\Delta_{AB} = \mathcal{P}_{A}- \mathcal{P}_{B}$ is sufficiently small. (The value $0.6$ appears in the formula as it is roughly the average value for $\mathcal{P}_{A}$ across the professional tours.)

This relation is one-to-one, implying that given a value of $\M_{AB}$, it may, up to some small error, be mapped to expected total number of games or margin for that match. In particular, this suggests that $\M_{AB}$ may serve as a strong covariate for these quantities. This observation suggests the possibility of modelling odds for total games and various other markets using odds for the win market. Consequently, after the coin toss is observed, one should in principle be able to update
the odds for totals and handicap markets by conditioning on the realised initial server. Noting that bookmaker odds for various relevant markets are never updated after the coin toss, this suggests a small but systematic inefficiency.
\end{remark}

\section*{Appendix}
Let \(\mathcal S^{A}_{nm}\) denote the probability that a set in which
\(A\) serves first ends \(n\)-\(m\) from \(A\)'s perspective. Similarly,
let \(\mathcal S^{B}_{nm}\) denote the corresponding probability when
\(B\) serves first. Following \cite[equations A.1-A.6]{NewKel}, define
\begin{align*} S_{60}(x_1,y_2) &= (x_1 y_2)^3\\ S_{61}(x_1,x_2,y_1,y_2) &= 3x_1^3 x_2 y_2^3 + 3x_1^4 y_1 y_2^2 \\ S_{62}(x_1,x_2,y_1,y_2) &= 12 x_1^3 x_2 y_1 y_2^3 + 6x_1^2 x_2^2 y_2^4 + 3x_1^4 y_1^2 y_2^2 \\ S_{63}(x_1,x_2,y_1,y_2) &= 24 x_1^3 x_2^2 y_1 y_2^3 + 24 x_1^4 x_2 y_1^2 y_2^2 + 4 x_1^2 x_2^3 y_2^4 + 4 x_1^5 y_1^3 y_2 \\ S_{64}(x_1,x_2,y_1,y_2) &= 60 x_1^3 x_2^2 y_1^2 y_2^3 + 40 x_1^2 x_2^3 y_1 y_2^4 + 20x_1^4 x_2 y_1^3 y_2^2 + 5 x_1 x_2^4 y_2^5 + x_1^5 y_1^4 y_2 \\ S_{75}(x_1,x_2,y_1,y_2) &= 100 x_1^3 x_2^3 y_1^2 y_2^4 + 100 x_1^4 x_2^2 y_1^3 y_2^3 + 25 x_1^2 x_2^4 y_1 y_2^5 + 25 x_1^5 x_2 y_1^4 y_2^2 + x_1 x_2^5 y_2^6 + x_1^6 y_1^5 y_2 \end{align*} 
Then $\S^A_{60} = S_{60}(\G_{A}, 1-\G_{B})$, $\S^A_{06} = S_{06}(1-\G_{A},\G_{B})$ and
\[
\mathcal S^A_{nm}
=
S_{nm}(\G_A,1-\G_A,\G_B,1-\G_B),
\qquad
\mathcal S^A_{mn}
=
S_{mn}(1-\G_A,\G_A,1-\G_B,\G_B),
\]
for $nm\in \{61,62,63,64,75\}$ and
\[
\mathcal S^A_{66}
=
1-
\left(
\sum_{m=0}^4 \mathcal S^A_{6m}
+
\sum_{m=0}^4 \mathcal S^A_{m6}
\right)
-\mathcal S^A_{75}
-\mathcal S^A_{57}.
\]
Hence, writing $\mathcal T_{A}^A$ for the probability that $A$ wins a tiebreak, as the initial server (see \cite[section~3.2]{NewKel}), 
\[
\mathcal S^A_{76}=\mathcal S^A_{66}\mathcal T_{A}^A,
\qquad
\mathcal S^A_{67}=\mathcal S^A_{66}(1-\mathcal T_{A}^A).
\]

\end{document}